\documentclass[12pt, reqno]{amsart}
\usepackage{amsmath, amsthm, amscd, amsfonts, amssymb, graphicx, color}
\usepackage[bookmarksnumbered, colorlinks, plainpages]{hyperref}

\makeatletter \oddsidemargin.9375in \evensidemargin \oddsidemargin
\newtheorem{theorem}{Theorem}[section]
\usepackage{titlesec}
\titleformat{\section}{\bfseries\Large}{\thesection}{1em}{}
\titleformat{\subsection}{\bfseries\normalsize}{\thesubsection}{1em}{}

\theoremstyle{definition}

\numberwithin{equation}{section}
\usepackage{amsthm}
\newtheoremstyle{boldquestion}  
{\topsep}                      
{\topsep}                      
{\normalfont}                 
{}                            
{\bfseries}                   
{.}                           
{ }                           
{\thmname{#1}\thmnumber{ #2}} 

\theoremstyle{boldquestion}
\newtheorem{question}{Question}
\usepackage{amsthm}

\newtheoremstyle{boldremark}  
{}{}                     
{\normalfont}              
{}                         
{\bfseries}                
{.}                        
{ }                       
{}

\theoremstyle{boldremark}
\newtheorem{remark}{Remark}

\usepackage{amsthm}

\makeatletter
\renewenvironment{proof}[1][\proofname]{%
	\par\pushQED{\qed}\normalfont
	\topsep6\p@\@plus6\p@\relax
	\trivlist
	\item[\hskip\labelsep\bfseries #1\@addpunct{.}]%
}{%
	\popQED\endtrivlist\@endpefalse
}
\makeatother

\begin{document}
\setcounter{page}{1}


\title[]{On Fair and Tolerant Colorings of Graphs}
\author[]{Saeed Shaebani}
\thanks{{\scriptsize
\hskip -0.4 true cm MSC(2020): Primary: 05C15; Secondary: 05C07.
 }}

\begin{abstract}
\noindent
A (not necessarily proper) vertex coloring of a graph $G$ with color classes $V_1$, $V_2$, $\dots$, $V_k$,
is said to be a {\it Fair And Tolerant vertex coloring of $G$ with $k$ colors}, whenever $V_1$, $V_2$, $\dots$, $V_k$ are nonempty and
there exist two real numbers $\alpha$ and $\beta$ such that
$\alpha \in [0,1]$ and $\beta \in [0,1]$ and the following condition holds for each arbitrary vertex $v$ and every arbitrary color class $V_i$:
$$
\bigl| V_i \cap N (v) \bigr| =
\begin{cases}
	\alpha \deg (v)  &  \mbox{ if }  \ \   v \notin V_i \\
	\beta \deg (v)   &  \mbox{ if }  \ \   v \in V_i . 
\end{cases}
$$
The {\it FAT chromatic number} of $G$, denoted by $\chi ^{{\rm FAT}} (G)$, is defined as the maximum positive integer $k$ for which $G$ admits a 
Fair And Tolerant vertex coloring with $k$ colors.
The concept of the FAT chromatic number of graphs was introduced and studied by Beers and Mulas, where they
asked for the existence of a function $f \colon  \mathbb{N} \to \mathbb{R}$
in such a way that the inequality
$\chi ^{{\rm FAT}} (G) \ \leq \ f \big( \chi (G) \big)$
holds for all graphs $G$.
Another similar interesting question concerns the existence of some function
$g \colon  \mathbb{N} \to \mathbb{R}$
such that the inequality
$\chi (G) \ \leq \ g \left( \chi ^{{\rm FAT}} (G) \right)$
holds for every graph $G$.
In this paper, we establish that both questions admit negative resolutions.
\end{abstract}

\maketitle
\section{Introduction}

In this paper, we investigate coloring parameters of graphs. 
Throughout the text, except when $[\,\cdot\,]$ indicates bibliographic references, 
the shorthand $[n]$ will denote the set $\{1,2,3,\dots,n\}$.

Our focus is on simple graphs whose vertex sets are nonempty and finite. 
Standard graph-theoretic notation is employed, following classical references such as \cite{Bollobas1998},
\cite{BondyMurty2008}, \cite{Diestel2017}, and \cite{West2001}.

For a graph $G$ and a vertex $v$ of $G$, the {\it neighborhood} of $v$, which is denoted by $N_G (v)$, is defined as
the set of all vertices that are adjacent to $v$. Each member of $N_G (v)$ is called a
{\it neighbor} of $v$. The size of $N_G (v)$ is called the {\it degree} of $v$, and it is denoted by $\deg _G (v)$.
Whenever there is no ambiguity on the graph $G$, and the graph 
$G$ is clear from context, we just write $N(v)$ and $\deg (v)$ instead
of $N_G (v)$ and $\deg_G (v)$, respectively.

\noindent
By a {\it vertex coloring} of a graph $G$, we mean a function
$c \colon V(G) \to C ,$
where $C$ is called the {\it set of colors}, and its elements are referred to as {\it colors}.
If $c(v) = r$, we say that the vertex $v$ is assigned the color $r$. 

\noindent
For each color $r$ in $C$, the set of all vertices $v$ for which $c(v) = r$, which is equal to $c^{-1} (r)$,
is said to be the {\it color class of the color} $r$.

\noindent
The set of nonempty color classes of a coloring $c \colon V(G) \to C$
provides a partition of $V(G)$.
Regardless of the specific names of the colors, this partition divides $V(G)$ into pairwise disjoint blocks,
each consisting of all vertices assigned the same color.
Thus, a partition of $V(G)$ may be naturally regarded as the set of nonempty color classes induced by some vertex coloring of 
$G$.

Let $G$ be a graph, $v$ be an arbitrary vertex of $G$, and $S$ be an arbitrary subset of $V(G)$.
The symbol $e (v,S)$ stands for the number of neighbors of $v$ that are also included in $S$. More precisely, $e (v,S)$
is equal to
$$e (v,S) := \bigl| S \cap N(v) \bigr| .$$

\noindent A coloring $c \colon  V(G) \to C $ is called {\it proper} if no two adjacent vertices are assigned the same color.
The {\it chromatic number} of $G$, denoted by $\chi (G)$, is defined as the minimum size of a set $C$ that a proper coloring
$c \colon  V(G) \to C $ exists.

Let $G$ be a graph, and $C$ be a set of size $k$.
Also, let $c \colon V(G) \to C$ be a (not necessarily proper) vertex coloring of $G$ with color classes $V_1$, $V_2$, $\dots$, $V_k$.
The coloring $c \colon V(G) \to C$ is called
a {\it Fair And Tolerant vertex coloring of $G$ with $k$ colors}, whenever $V_1$, $V_2$, $\dots$, $V_k$ are nonempty and
there exist two real numbers $\alpha$ and $\beta$ such that
$\alpha \in [0,1]$ and $\beta \in [0,1]$ and the following condition holds for each arbitrary vertex $v$ and every arbitrary color class $V_i$:
$$
e \left( v, V_i \right) =
\begin{cases}
	\alpha \deg (v)  &  \mbox{ if }  \ \   v \notin V_i \\
	\beta \deg (v)   &  \mbox{ if }  \ \   v \in V_i . 
\end{cases}
$$

\medspace

\noindent
The concept of Fair And Tolerant vertex colorings was introduced and investigated by Beers and Mulas in \cite{beers20252},
and for brevity they used the abbreviated phrase “{\it FAT Colorings}”.
Also, the abbreviated phrase “{\it FAT $k$-coloring}” was used in \cite{beers20252} instead of "Fair And Tolerant vertex coloring with $k$ colors".

\medskip

Let $G$ be a graph with at least one vertex $v$ of nonzero degree. Suppose that $G$ admits a FAT $k$-coloring
with color classes $V_1$, $V_2$, $\dots$, $V_k$ and corresponding parameters
$\alpha$ and $\beta$.
Then we have:
$$\deg (v) = \displaystyle\sum _{i=1} ^k  e \bigl( v , V_i\bigr) =	\beta \deg (v) + (k-1) \alpha \deg (v) ;$$
and therefore, the equality
$$\beta + (k-1) \alpha = 1$$
holds \cite{beers20252} as a relation between $k$ and $\alpha$ and $\beta$.

\medskip

For every graph $G$, a FAT $1$-coloring is obtained by setting $\alpha = 0$ and $\beta = 1$, and taking the entire vertex set $V(G)$
as the sole color class. Therefore, the set
$$\bigl\{ k \colon \ \mbox{The graph } G \mbox{ admits a FAT } k \mbox{-coloring} \bigr\}$$
is nonempty. Also, this set is bounded above, because the number of color classes in every FAT coloring of $G$ never exceeds $|V(G)|$.
Therefore, this set has a maximum value, which is called the {\it FAT chromatic number} of $G$, and it is denoted by $\chi ^{{\rm FAT}} (G)$.
More precisely,
$$\chi ^{{\rm FAT}} (G) \  := \ \max \bigl\{ k \colon \ \mbox{The graph } G \mbox{ admits a FAT } k \mbox{-coloring} \bigr\} .$$
The concept of the FAT chromatic number of graphs was introduced and studied by Beers and Mulas \cite{beers20252} in 2025.

\medskip

\begin{remark}
	In the definition of FAT colorings, it is required that all color classes be nonempty
	to ensure that the FAT chromatic number is well-defined. Otherwise, if empty color classes were permitted in FAT colorings,
	then for any arbitrary graph $G$ and any positive integer $k$,
	the monochromatic coloring
	$c \colon V(G) \to [k]$
	that assigns color $1$ to every vertex of $G$ would be a FAT $k$-coloring with corresponding parameters  
	$\alpha = 0$ and $\beta = 1$. Hence, a FAT $k$-coloring would exist for every 
	$k$, and the notion of FAT chromatic number would be rendered meaningless.
\end{remark}

It was shown by Beers and Mulas \cite{beers20252} that there exist some graphs $G_1$, $G_2$, and $G_3$
in such a way that the following three conditions hold:
$$\begin{array}{lcccr}
	     \chi ^{{\rm FAT}} \left(  G_1  \right)  <   \chi \left( G_1 \right)  ,  &
	         &
	     \chi ^{{\rm FAT}} \left(  G_2  \right)  =   \chi \left( G_2 \right)  ,  &
	         &    
	     \chi ^{{\rm FAT}} \left(  G_3  \right)  >   \chi \left( G_3 \right)  .
\end{array}$$
Therefore, the parameters $\chi ^{{\rm FAT}} (G)$ and $\chi(G)$ may coincide, or one may be strictly greater or strictly smaller than the other.
So, the following interesting open questions were raised in \cite{beers20252} by Beers and Mulas.

\medskip

\begin{question} \label{Q1}
	What is the maximum possible gap between $\chi ^{{\rm FAT}} $ and $\chi $?
\end{question}

\medskip

\begin{question} \label{Q2}
	Is $\chi ^{{\rm FAT}} $ bounded in terms of $\chi $? More precisely, does there exist a function $f \colon  \mathbb{N} \to \mathbb{R}$
	in such a way that the inequality
	$$\chi ^{{\rm FAT}} (G) \ \leq \ f \big( \chi (G) \big)$$
	holds for all graphs $G$?
\end{question}

\medskip

\noindent
A similar question could also be raised as follows.

\medskip

\begin{question} \label{Q3}
	Is $\chi$ bounded in terms of $\chi ^{{\rm FAT}}  $? In other words, does there exist a function
	$g \colon  \mathbb{N} \to \mathbb{R}$
	such that the inequality
	$$\chi (G) \ \leq \ g \bigg( \chi ^{{\rm FAT}} (G) \bigg)$$
	holds for every graph $G$?
\end{question}

\medskip

\begin{question} \label{Q4}
	Does there exist a class of graphs $\mathcal{G}$ for which $\chi (G)$ is bounded over $\mathcal{G}$
	while $\chi ^{{\rm FAT}} (G)$ diverges when $G$ ranges over $\mathcal{G}$?
	
\end{question}

\medskip

\begin{question} \label{Q5}
	Does there exist a class of graphs $\mathcal{G}$ such that $\chi ^{{\rm FAT}} (G)$ is bounded over $\mathcal{G}$
	whereas $\chi (G)$ diverges when $G$ ranges over $\mathcal{G}$?
\end{question}

\medskip

\noindent
In this paper, we aim to answer these 5 questions.

\section{Main Results}

This section is devoted to addressing Questions \ref{Q1}, \ref{Q2}, \ref{Q3}, \ref{Q4}, and \ref{Q5}.
To this end, we divide the discussion into two cases: disconnected graphs and connected graphs.

\subsection{The Case of Disconnected Graphs}

In this subsection, we aim to answer Questions \ref{Q1}, \ref{Q2}, \ref{Q3}, \ref{Q4}, and \ref{Q5}
for the class of disconnected graphs.
In this regard, we provide the following two theorems.

\begin{theorem}  \label{Disconnected1}
	For every two positive integers $L_1$ and $L_2$ with $L_1 < L_2$, there exists a disconnected graph $G$ for which
	$\chi (G) = L_1$ and $\chi ^{{\rm FAT}} (G) = L_2$.
\end{theorem}

\begin{proof}
	First, we prove the theorem for the case where $L_1 = 1$.
	
	\noindent
	Let us regard the edgeless graph $\bar{K}_{L_2}$ with $L_2$ vertices whose vertex set equals
	$\left\{ x_1 , x_2 , \dots , x_{L_2} \right\}$. This graph is the complement of the complete graph $K_{L_2}$.
	
	\noindent
	The chromatic number of $\bar{K}_{L_2}$ is equal to $1$.
	Now the singletons $\left\{ x_1 \right\}$, $\left\{ x_2 \right\}$, $\dots$, $\left\{ x_{L_2} \right\}$
	constitute the color classes of a FAT $L_2$-coloring of $\bar{K}_{L_2}$ with corresponding parameters
	$\alpha = 0$ and $\beta = 1$. Therefore, $\chi ^{{\rm FAT}} \big( \bar{K}_{L_2} \big) \geq L_2 $.
	
	\noindent
	Since the definition of the notion of FAT coloring stipulates that every color class is nonempty, it follows that
	the number of color classes in every FAT coloring of $\bar{K}_{L_2}$ is less than or equal to the number of vertices
	of $\bar{K}_{L_2}$. Therefore,
	$\chi ^{{\rm FAT}} \big( \bar{K}_{L_2} \big) \leq \left| V \left( \bar{K}_{L_2} \right)  \right| = L_2 $.
	Accordingly, $\chi \big( \bar{K}_{L_2} \big) = 1$ and $\chi ^{{\rm FAT}} \big( \bar{K}_{L_2} \big) = L_2 $
	for each positive integer $L_2$ which is greater than $1$.
	
	Having treated the case $L_1 = 1$, we now assume 
	$L_1 \geq 2$
	for the rest of the proof.
	
	Let us consider $L_2$ pairwise disjoint sets $S_1$, $S_2$, $\dots$, $S_{L_2}$, each of size $L_1$, as follows:
	$$\begin{array}{lccr}
		S_i := \left\{ x ^{(i)} _1  ,   x ^{(i)} _2  ,  \dots ,  x ^{(i)} _{L_1}    \right\} &   & \ {\rm for}   &   i=1,2, \dots , L_2 .
	\end{array}$$
	\medskip
	Now, define a graph $G$ with $L_1 L_2$ vertices whose vertex set equals
	$$ V \left( G \right) := S_1 \cup S_2 \cup \dots \cup S_{L_2} ; $$
	and its edge set equals
	$$ E (G) := \left\{   x ^{(i)} _j  x ^{(i)} _k  \colon \ \ i \in \left[ L_2 \right]  \ {\rm and} \ 1 \leq j < k \leq L_1 \right\} . $$
    Indeed, the graph $G$ is a union of $L_2$ vertex disjoint cliques, each of order $L_1$.
    
    \noindent
    Since $G$ is a union of $L_2$ vertex disjoint cliques, each with chromatic number $L_1$, it follows that the chromatic number of $G$ is also $L_1$.
    
    Now, we concern the FAT chromatic number of $G$.
    If we assign the color $i$ to all vertices of $S_i$ for $i = 1 , 2 , \dots , L_2$, then
    $S_1$, $S_2$, $\dots$, $S_{L_2}$
    constitute the color classes of a FAT $L_2$-coloring of $G$ with corresponding parameters $\alpha = 0$ and $\beta = 1$.
    Accordingly,
    $$\chi ^{{\rm FAT}} (G) \geq L_2 .$$
    Our goal is to prove that $\chi ^{{\rm FAT}} (G) = L_2 .$
    
    \noindent
    Put $\chi ^{{\rm FAT}} (G) = k$, and consider some FAT $k$-coloring of $G$ with $k$ nonempty
    color classes
    $T_1$, $T_2$, $\dots$, $T_k$, together with some corresponding parameters
    $\alpha ' \in [0,1]$ and $\beta ' \in [0,1]$.
    Since $\chi ^{{\rm FAT}} (G) \geq L_2 $, it suffices to show that $k \leq L_2$.    
    \noindent
    We proceed by considering two cases.
    
    \medskip
    
    \noindent
    {\bf Case I:} The case where $S_1 \cap T_b \neq \varnothing$ for every $b \in [k]$.
    
    \noindent
    In this case, each of $T_1$, $T_2$, $\dots$, $T_k$ intersects $S_1$. Hence, the number of color classes
    must be less than or equal to the number of vertices in $S_1$. Thus, $k \leq \left| S_1 \right| = L_1$.
    Now, the condition $L_1 < L_2$ implies $k < L_2$; and we are done in this case.
    
    \medskip
    
    \noindent
    {\bf Case II:} The case where $S_1 \cap T_b = \varnothing$ for some $b \in [k]$.
    
    \noindent
    In this case, since $x ^{(1)} _1 \in S_1$, all neighbors of $x ^{(1)} _1$ also belong to the clique $S_1$. Therefore,
    the condition $S_1 \cap T_b = \varnothing$ implies
    $$ e \left(  x ^{(1)} _1  ,  T_b  \right) = 0. $$
    On the other hand, because of $ x ^{(1)} _1 \notin T_b $, it follows that
    $$ e \left(  x ^{(1)} _1  ,  T_b  \right) = \alpha ' \ldotp \deg_{G} \left(  x ^{(1)} _1  \right) = \alpha ' \ldotp \left( L_1  -  1 \right) .$$
    Therefore,
    $$\alpha ' \ldotp \left( L_1  -  1 \right) = 0 .$$
    Now, since $L_1 \geq 2$, we conclude that
    $$\alpha ' = 0 .$$
    Also, by the condition
    $\beta '  +  (k-1) \ldotp \alpha ' = 1$, one finds that
    $$\beta '  =  1  .$$
    For each arbitrary $ S_i \in \left\{ S_1 , S_2 , \dots , S_{L_2} \right\}$,
    define
    $$a_i := \min \left\{ j \colon \ \ j \in [k]  \  {\rm and}  \  S_i \cap T_j \neq \varnothing  \right\} ;$$
    and consider some vertex $z$ in $S_i \cap T_{a_i}$. Since $z \in T_{a_i}$, it follows that
    $$ e \left( z , T_{a_i}\right) = \beta ' \ldotp \deg _G (z) = \deg _G (z) .$$
    Thus, all neighbors of $z$ must also belong to the $T_{a_i}$. Accordingly,
    $$S_i \subseteq T_{a_i} .$$
    This shows that for each arbitrary clique $ S_i \in \left\{ S_1 , S_2 , \dots , S_{L_2} \right\}$,
    there is a unique color class in $\left\{ T_1 , T_2 , \dots , T_k \right\}$ that intersects $S_i$; and besides,
    the clique $S_i$ is entirely contained in that unique color class.
    Moreover, every arbitrary color class $T_j$ is equal to the union of all
    cliques $S_{\ell}$ for which $S_{\ell} \cap  T_j \neq \varnothing$.
    Therefore, because $T_1$, $T_2$, $\dots$, $T_k$ are nonempty, 
    the function
    $$\begin{array}{rcl}
    	 f \colon \ \left\{ S_1 , S_2 , \dots , S_{L_2} \right\} & \longrightarrow & \left\{ T_1 , T_2 , \dots , T_k \right\} \\
    	                             S_i                        & \longmapsto     &              T_{a_i}
    \end{array}$$
    is well-defined and surjective. Accordingly, $k \leq L_2$, which completes the proof.
    
\end{proof}

\begin{theorem}  \label{Disconnected2}
	For every two positive integers $L_1$ and $L_2$ with $1 < L_1 < L_2$, there exists a disconnected graph $G$ such that
	$\chi ^{{\rm FAT}} (G) = L_1$ and $\chi (G) = L_2$.
\end{theorem}

\begin{proof}
	Let $S_1$, $S_2$, $\dots$, $S_{L_1 - 1}$ denote $L_1 - 1$ pairwise disjoint sets of size $L_1$, as follows:
	$$\begin{array}{lccr}
		S_i := \left\{ x ^{(i)} _1  ,   x ^{(i)} _2  ,  \dots ,  x ^{(i)} _{L_1}    \right\} &   & \ {\rm for}   &   i=1,2, \dots , L_1 - 1 .
	\end{array}$$
	In addition, consider another set $S_{L_1}$, disjoint from all of $S_1$, $S_2$, $\dots$, $S_{L_1 - 1}$, whose cardinality is $L_2$, given by:
	$$S_{L_1} := \left\{ x ^{\left(L_1\right)} _1  ,   x ^{\left(L_1\right)} _2  ,  \dots ,  x ^{\left(L_1\right)} _{L_2}    \right\} .$$
	
	\noindent
	Now, regard a graph $G$ with vertex set
	$$ V \left( G \right) := S_1 \cup S_2 \cup \dots \cup S_{L_1} ; $$
	and the edge set $E(G) := E_1 \cup E_2$, where $E_1$ and $E_2$ are defined as the following:
	\begin{itemize}
		\item $ E_1 := \left\{   x ^{(i)} _j  x ^{(i)} _k  \colon \ \ i \in \left[ L_1 - 1 \right]  \ {\rm and} \ 1 \leq j < k \leq L_1 \right\} , $
		\medskip
		\item $E_2 := \left\{   x ^{\left(L_1\right)} _j  x ^{\left(L_1\right)} _k  \colon \ \  1 \leq j < k \leq L_2 \right\} .$
	\end{itemize}
	
	\noindent
	The graph $G$ is defined as the union of $L_1$ pairwise disjoint complete graphs. Among these, the first $L_1 -1$ complete graphs have chromatic number $L_1$,
	whereas the $L_1$-th complete graph has chromatic number $L_2$. Therefore, the chromatic number of $G$ is equal to
	$$\chi (G)  = \max \left\{ L_1 , L_2 \right\} = L_2 .$$
	
	By assigning color $i$ to all vertices of $S_i$ for $i = 1, 2, \dots, L_1$, we obtain a FAT $L_1$-coloring of $G$ with parameters $\alpha = 0$ and $\beta = 1$.
	This construction immediately implies
	$$ \chi^{{\rm FAT}}(G) \geq L_1 . $$
	We now demonstrate that this inequality is in fact an equality, that is, $\chi ^{{\rm FAT}} (G) = L_1 .$
	
	\noindent
	Set $\chi^{{\rm FAT}}(G) = k$, and suppose that $G$ admits a FAT $k$-coloring with $k$ nonempty color classes
	$T_1$, $T_2$, $\dots$, $T_k$, accompanied by parameters $\alpha' \in [0,1]$ and $\beta' \in [0,1]$.
	Given that $\chi^{{\rm FAT}}(G) \geq L_1$, our task reduces to proving that $k \leq L_1$.
	    
	\noindent
	The proof is carried forward by a case analysis, divided into two parts.
	
	\medskip
	
	\noindent
	{\bf Case I:} Suppose that $S_1 \cap T_b \neq \varnothing$ for every $b \in [k]$.
	
	\noindent
	In this situation, each of the color classes $T_1$, $T_2$, $\dots$, $T_k$ intersects $S_1$.
	Consequently, the number of color classes cannot exceed the number of vertices contained in $S_1$. 
	Therefore, $k \leq \left| S_1 \right| = L_1$.
	
	\medskip
	
	\noindent
	{\bf Case II:} Assume that $S_1 \cap T_b = \varnothing$ for some $b \in [k]$.
	
	\noindent
	The vertex $x^{(1)}_1$ and all of its neighbors lie within $S_1$.
	Hence, the condition $S_1 \cap T_b = \varnothing$ forces
	$$ e \left(  x ^{(1)} _1  ,  T_b  \right) = 0. $$
	On the other hand, because $x^{(1)}_1 \notin T_b$, the definition of a FAT coloring yields
	$$ e \left(  x ^{(1)} _1  ,  T_b  \right) = \alpha ' \ldotp \deg_{G} \left(  x ^{(1)} _1  \right) = \alpha ' \ldotp \left( L_1  -  1 \right) .$$
	Thus,
	$$\alpha ' \ldotp \left( L_1  -  1 \right) = 0 .$$
	Since $L_1 \geq 2$, we infer that
	$$ \alpha' = 0 .$$
	Moreover, the relation $\beta' + (k-1)\alpha' = 1$ immediately gives
	$$\beta' = 1 . $$
	Now, fix any clique $S_i \in \left\{ S_1, S_2, \dots, S_{L_1} \right\}$ and define
	$$ a_i := \min \left\{ j \in [k] \ \colon \ \  S_i \cap T_j \neq \varnothing \right\}. $$
	Choose a vertex $z \in S_i \cap T_{a_i}$. Since $z \in T_{a_i}$, we obtain
	$$ e \left( z , T_{a_i}\right) = \beta ' \ldotp \deg _G (z) = \deg _G (z) .$$
	Consequently, all neighbors of $z$ must also lie in $T_{a_i}$, and therefore,
	$$	S_i \subseteq T_{a_i} . $$
	This shows that each clique $S_i$ is entirely contained in a unique color class $T_{a_i}$. Conversely,
	every color class $T_j$ is the union of those cliques $S_\ell$ intersecting it. In other words,
	every $T_j$ equals the union of all cliques $S_{\ell}$ with $S_{\ell} \cap  T_j \neq \varnothing$.
	Since $T_1$, $T_2$, $\dots$, $T_k$ are nonempty, the mapping
	$$\begin{array}{rcl}
		f \colon \ \left\{ S_1 , S_2 , \dots , S_{L_1} \right\} & \longrightarrow & \left\{ T_1 , T_2 , \dots , T_k \right\} \\
		S_i                        & \longmapsto     &              T_{a_i}
	\end{array}$$
	is well-defined and surjective. Hence, $k \leq L_1$, completing the argument.
\end{proof}

\medskip

Now we turn to Questions \ref{Q1}, \ref{Q2}, \ref{Q3}, \ref{Q4}, and \ref{Q5}, 
considered within the class of disconnected graphs. 
The conclusions are summarized in the following Remark.

\medskip

\begin{remark}
	Theorem \ref{Disconnected1} shows that for every arbitrary positive integer $L$,
	there exists a sequence $\left\{ G_n \right\}_{n=1} ^{\infty}$ of disconnected graphs
	for which $\chi \left( G_n \right) = L$ for each positive integer $n$, while
	the sequence $\chi^{{\rm FAT}} \left( G_n \right)$ diverges to $+\infty$ as $n$ tends to infinity.
	Thus, Question \ref{Q4} is affirmed, while Question \ref{Q2} receives a negative answer.
	In addition, Theorem \ref{Disconnected2} implies that for every arbitrary positive integer $L > 1$,
	there exists a sequence $\left\{ H_n \right\}_{n=1} ^{\infty}$ of disconnected graphs
	such that $\chi^{{\rm FAT}} \left( H_n \right) = L$ for all positive integers $n$, whereas
	the sequence $\chi \left( H_n \right)$ diverges to $+\infty$ as $n$ tends to infinity.
    Consequently, Question \ref{Q5} admits an affirmative resolution, whereas Question \ref{Q3} is negated.
    Moreover, in answer to Question \ref{Q1}, the sequences $\left\{ G_n \right\}_{n=1} ^{\infty}$ and $\left\{ H_n \right\}_{n=1} ^{\infty}$ 
    demonstrate that both differences $\chi^{{\rm FAT}}(G) - \chi(G)$ and $\chi(H) - \chi^{{\rm FAT}}(H)$ 
    remain unbounded, even within the class of disconnected graphs.
\end{remark}

\medskip

Let $L_1$ and $L_2$ be two arbitrary positive integers with $L_1 < L_2$.
Theorem \ref{Disconnected1} asserts that there exists a disconnected graph
$G_1$ satisfying $\chi \left( G_1 \right) = L_1$ and $\chi ^{{\rm FAT}} \left( G_1 \right) = L_2$.
Moreover, Theorem \ref{Disconnected2} guarantees the existence of some disconnected graph
$G_2$ with $\chi ^{{\rm FAT}} \left( G_2 \right) = L_1$ and $\chi \left( G_2 \right) = L_2$, provided that
$L_1 > 1$.
Now, the following natural Question arises.

\medskip
	
\begin{question} \label{Q6}
	Does Theorem \ref{Disconnected2} remain valid without the condition $L_1 > 1$?
\end{question}

\medskip

Question \ref{Q6} is answered in the negative in the following Remark.
	
\begin{remark}	
	Let $G$ be any arbitrary disconnected graph consisting of the components $G_1$, $G_2$, $\dots$, $G_{\ell}$.
	If we assign the color $i$ to all vertices of $G_i$ for $i = 1 , 2 , \dots , \ell$, then
	$G_1$, $G_2$, $\dots$, $G_{\ell}$
	constitute the color classes of a FAT $\ell$-coloring of $G$ with corresponding parameters $\alpha = 0$ and $\beta = 1$.
	Consequently, $\chi ^{{\rm FAT}} (G) \geq \ell$. We conclude that the FAT chromatic number of every disconnected graph is
	strictly greater than $1$. This shows that the condition $L_1 > 1$ is necessary in Theorem \ref{Disconnected2}.
\end{remark}

\subsection{The Case of Connected Graphs}

In this subsection we address Questions \ref{Q1}, \ref{Q2}, \ref{Q3}, \ref{Q4}, and \ref{Q5}, for the class of connected graphs.
To this end, we state the following theorem.

\begin{theorem} \label{Connected}
	For each odd positive integer $n \geq 5$, there exist some connected graphs $G_1$
	and $G_2$ for which
	$$\begin{array}{cccc}
		\chi \left( G_1 \right) =2 , \  &  \chi ^{{\rm FAT}} \left( G_1 \right) = n, \ & \chi \left( G_2 \right) =n,\  & \chi ^{{\rm FAT}} \left( G_2 \right) = 2. 
	\end{array}$$ 
\end{theorem}

\begin{proof}
	Let $A:= \left\{ x_1 , x_2 , \dots , x_n \right\}$ and $B:= \left\{ y_1 , y_2 , \dots , y_n \right\}$.
	Also, let $G_1$ be a graph with vertex set $V \left( G_1 \right) := A \cup B$
	whose edge set equals
	$$E \left( G_1 \right) := \bigl\{ x_i y_j \ \colon \ \   i,j \in [n] \mbox{ and } i\neq j \bigr\} .$$
	
	\noindent
	Indeed, $G_1$ is a connected bipartite graph which is obtained by removing the edges of the
	perfect matching
	$\bigl\{ x_1 y_1 , x_2 y_2 , \dots , x_n y_n \bigr\}$
	from the complete bipartite graph $K_{n,n}$ with one part $A$ and the other part $B$.
	
	\medskip
	
	\noindent
	The bipartite graph $G_1$ is an $(n-1)$-regular graph with chromatic number $\chi \left( G_1 \right) =2 .$
	
	\noindent
	It was shown in \cite{beers20252} that the FAT chromatic number of every connected graph with minimum degree $\delta$ never
	exceeds $\delta + 1$. Now, 
	since $\delta \left( G_1 \right) = n-1$, it follows that
	$$\chi ^{{\rm FAT}} \left( G_1 \right) \leq \delta  \left( G_1 \right) + 1 = n .$$
	Now, in order to show that $\chi ^{{\rm FAT}} \left( G_1 \right) = n ,$ it suffices to find a FAT coloring of $G_1$
	with exactly $n$ color classes. For each $i$ in $[n]$, we assign the color $i$ to the vertices $x_i$ and $y_i$; forming the $i$-th
	color class
	$$C_i := \bigl\{ x_i , y_i \bigr\} .$$
	For each vertex $z \in V \left( G_1 \right)$ and every color class $C_i \in \bigl\{ C_1 , C_2 , \dots , C_n \bigr\}$ we have:
	$$
	e \left( z , C_i\right) =
	\begin{cases}
		1   &  \mbox{ if }  \ \   z \notin C_i \\
		0   &  \mbox{ if }  \ \   z \in C_i . 
	\end{cases}
	$$
	Hence, by setting
	$$\begin{array}{lcccr}
		\alpha _{G_1} := \frac{1}{n-1}  &     &   {\rm and}    &     &   \beta _{G_1} := 0 , 
	\end{array}$$
	and the fact that $\deg _{G_1} (z) = n-1$, one finds that
	$$\begin{array}{lcccr}
		\alpha _{G_1} . \deg _{G_1} (z) = 1    &   &    {\rm and}  &     &   \beta _{G_1} . \deg _{G_1} (z) = 0 . 
	\end{array}$$
	Accordingly, $\beta _{G_1} + (n-1) \alpha _{G_1} = 1$ and
	for each vertex $z \in V \left( G_1 \right)$ and every color class $C_i \in \bigl\{ C_1 , C_2 , \dots , C_n \bigr\}$
	the condition
	$$
	e \left( z , C_i\right) =
	\begin{cases}
		\alpha _{G_1} . \deg _{G_1} (z)   &  \mbox{ if }  \ \   z \notin C_i \\
		\beta _{G_1} . \deg _{G_1} (z)    &  \mbox{ if }  \ \   z \in C_i  
	\end{cases}
	$$
	holds. Thus, $C_1$, $C_2$, $\dots$, $C_n$ are the color classes of a FAT $n$-coloring
	of $G_1$.
	We conclude that $\chi ^{{\rm FAT}} \left( G_1 \right) = n .$
	
	Thus far, we have constructed the connected graph $G_1$
	in such a way that
	$$\begin{array}{lcccr}
		\chi \left( G_1 \right) = 2  &     &   {\rm and}    &     &   \chi ^{{\rm FAT}} \left( G_1 \right) = n . 
	\end{array}$$
	
	\medskip
	
	Following the completion of the earlier stage, we are prepared to undertake the construction of graph $G_2$.
	
	\noindent
	We define three sets $V_1$, $V_2$, and $V_3$, as follows:
	\medskip
	\medskip
	\begin{itemize}
		\item $V_1 := \bigl\{ w_1 , w_2 , \dots , w_n \bigr\}$;
		\medskip
		\medskip
		\item $V_2 := \biggl\{ u ^{(i)} _j \colon \ \ i \in [n] \mbox{ and } j\in \left[ \frac{n-1}{2} \right]   \biggr\}$;
		\medskip
		\item $V_3 := \biggl\{ v ^{(i)} _j \colon \ \ i \in [n] \mbox{ and } j\in \left[ \frac{n-1}{2} \right]   \biggr\}$.
	\end{itemize}
	\medskip
	Now, the vertex set of $G_2$ is defined as $V \left( G_2 \right) := V_1 \cup V_2 \cup V_3$.
	In order to define the edge set of $G_2$, first we consider some matchings $M_1$, $M_2$, $\dots$, $M_n$, each of size
	$\frac{n-1}{2}$. 
	
	\noindent
	For each $i \in \left[ n \right]$, the matching $M_i$ equals
	$$M_i := \biggl\{   u ^{(i)} _1 v^{(i)} _1  , u ^{(i)} _2 v^{(i)} _2  , \dots ,  u ^{(i)} _{\frac{n-1}{2}} v^{(i)} _{\frac{n-1}{2}}   \biggr\} .$$
	Now, the edge set of $G_2$ equals
	$E \left( G_2 \right) := E_1 \cup E_2 \cup E_3 \cup E_4$, where $E_1$, $E_2$, $E_3$, and $E_4$ are defined as follows:
	\medskip
	\begin{itemize}
		\item $E_1 := \bigl\{ w_i w_j \colon \ \ 1 \leq i < j \leq n \bigr\}$;
		\medskip
		\medskip
		\item $E_2 := \biggl\{ w_i u ^{(i)} _j \colon \ \ i \in [n] \mbox{ and } j\in \left[ \frac{n-1}{2} \right]   \biggr\}$;
		\medskip
		\item $E_3 := \biggl\{ w_i v ^{(i)} _j \colon \ \ i \in [n] \mbox{ and } j\in \left[ \frac{n-1}{2} \right]   \biggr\}$;
		\medskip
		\medskip
		\item $E_4 := M_1 \cup M_2 \cup \dots \cup M_n$. 
	\end{itemize}
	\medskip
	Formally, $G_2$ is obtained from the complete graph $K_n$ on vertices $w_1,w_2,\dots,w_n$ by attaching to each vertex $w_i$ exactly $\tfrac{n-1}{2}$ pendant triangles
	$$
	w_i u^{(i)}_j v^{(i)}_j \qquad \ \mbox{ for all }j = 1 , 2 , \dots , \tfrac{n-1}{2} ;
	$$
	each such triangle consisting of the edges $w_i u^{(i)}_j$, $w_i v^{(i)}_j$, and $u^{(i)}_j v^{(i)}_j$.
    
    \medskip
    
    For each $i \in [n]$ and every $j \in \big[ \frac{n-1}{2}\big]$, we have
    $$\begin{array}{lcccr}
    	\deg _{G_2} \left( w_i \right) = 2(n-1)  &   &  {\rm and}  &   &   \deg _{G_2} \left(  u_j ^{(i)}  \right) =  \deg _{G_2} \left(  v_j ^{(i)}  \right) = 2 .
    \end{array}$$	
	
	Let us put
	$$\begin{array}{lllllllccr}
		S_1 := V_1,  &  &  S_2 := V_2 \cup V_3 ,  &  &  \alpha_{G_2} := \frac{1}{2}, &  &  {\rm and}  &  &  \beta_{G_2} := \frac{1}{2} .
	\end{array}$$
	Now, for each vertex $z \in S_1$, we have
	\medskip
	\begin{itemize}
		\item $e \left( z , S_1 \right) = n-1 = \frac{1}{2} \deg _{G_2} (z) = \beta  _{G_2} . \deg _{G_2} (z)$,
		\medskip
		\item $e \left( z , S_2 \right) = n-1 = \frac{1}{2} \deg _{G_2} (z) = \alpha _{G_2} . \deg _{G_2} (z)$. 
	\end{itemize}
	\medskip
	Also, for each vertex $z \in S_2$, we have
	\medskip
	\begin{itemize}
		\item $e \left( z , S_1 \right) = 1 = \frac{1}{2} \deg _{G_2} (z) = \alpha _{G_2} . \deg _{G_2} (z)$,
		\medskip
		\item $e \left( z , S_2 \right) = 1 = \frac{1}{2} \deg _{G_2} (z) = \beta  _{G_2} . \deg _{G_2} (z)$. 
	\end{itemize}
	\medskip
	We conclude that for each vertex $z \in V\left( G_2 \right)$ and every $S_i \in \left\{ S_1 , S_2 \right\}$, we have
	\medskip
	$$
	e \left( z , S_i\right) =
	\begin{cases}
		\alpha _{G_2} . \deg _{G_2} (z)   &  \mbox{ if }  \ \   z \notin S_i \\
		\beta  _{G_2} . \deg _{G_2} (z)   &  \mbox{ if }  \ \   z \in S_i ;
	\end{cases}
	$$
	and therefore, $S_1$ and $S_2$ provide the color classes of a FAT $2$-coloring of $G_2$.
	Accordingly,
		$$ \chi ^{{\rm FAT}} \left( G_2 \right) \geq  2 . $$
	Since $G_2$ is connected, its FAT chromatic number is less than or equal to $\delta \left( G_2 \right) +1$. Therefore,
	$$ \chi ^{{\rm FAT}} \left( G_2 \right) \leq  \delta \left( G_2 \right) +1 = 2 + 1 = 3 . $$
	Now, we claim that
	$ \chi ^{{\rm FAT}} \left( G_2 \right) = 2 . $
	Suppose, contrary to our claim, that
	$ \chi ^{{\rm FAT}} \left( G_2 \right) \neq 2 . $
	So, $ \chi ^{{\rm FAT}} \left( G_2 \right) = 3  $,
	and $G_2$ admits some FAT $3$-coloring with three color classes, say $T_1$, $T_2$, and $T_3$, together with corresponding
	parameters $\alpha$ and $\beta$.
	On account of the definition of FAT colorings, the color classes $T_1$, $T_2$, and $T_3$
	must be nonempty. Hence, the connectedness of $G_2$ implies that for two distinct color classes, say $T_1$ and $T_2$,
	there exists some edge of $G_2$, say $t_1 t_2$, in such a way that $t_1 \in T_1$ and $t_2 \in T_2$.
	Therefore,
	$$\begin{array}{lcccr}
		e \left( t_1 , T_2 \right) \neq 0, &  &  {\rm and}  &  &  e \left( t_1 , T_2 \right) =\alpha . \deg _{G_2} \left( t_1 \right) .
	\end{array}$$
	Hence, $\alpha$ must be positive.
	
	\noindent
	Let $i$ be an arbitrary element of $[n]$ and $j$ be an arbitrary element of $\left[ \frac{n-1}{2} \right]$.
	We may assume, without loss of generality, that
	$u _j ^{(i)} \in T_1$. Since the degree of 
	$ u_j ^{(i)} $ is $2$, one finds that
	$$e \left(  u_j ^{(i)} , T_2  \right) =  e \left(  u_j ^{(i)} , T_3  \right) =  \alpha . \deg _{G_2} \left(  u_j ^{(i)}  \right) =  2  \alpha > 0 .$$
	Thus, the vertex $ u_j ^{(i)} $ is adjacent to precisely one vertex of $T_2$ and precisely one vertex of $T_3$.
	Therefore, $2 \alpha = 1$; which implies $\alpha = \frac{1}{2}$.
	Since $\beta + (3-1) \alpha = 1$, we find that $\beta = 0$.
	This shows that no two distinct vertices in $\left\{ w_1 , w_2 , \dots , w_n \right\}$
	belong to the same color class.
	Since all vertices lie in the union of the three color classes $T_1$ and $T_2$ and $T_3$,
	it would follow that 
	$n \leq 3$; which contradicts the fact that $n\geq 5$.
	So, we conclude that
	$ \chi ^{{\rm FAT}} \left( G_2 \right) = 2  $;
	as desired.
\end{proof}

\medskip

We now proceed to examine Questions \ref{Q1}, \ref{Q2}, \ref{Q3}, \ref{Q4}, and \ref{Q5}, under the restriction to connected graphs.
The conclusions of this analysis are condensed in the subsequent Remark.

\medskip

\begin{remark}
	Theorem \ref{Connected} furnishes a sequence $\left\{ G_n \right\}_{n=1} ^{\infty}$ of connected graphs
	for which $\chi \left( G_n \right) = 2$ for each positive integer $n$, while
	the sequence $\chi^{{\rm FAT}} \left( G_n \right)$ diverges to $+\infty$ as $n$ tends to infinity.
    Thus Question \ref{Q4} is affirmed, whereas Question \ref{Q2} is answered negatively.
	Also, Theorem \ref{Connected} produces a sequence $\left\{ H_n \right\}_{n=1} ^{\infty}$ of connected graphs
	such that $\chi^{{\rm FAT}} \left( H_n \right) = 2$ for all positive integers $n$, whereas
	the sequence $\chi \left( H_n \right)$ diverges to $+\infty$ as $n$ tends to infinity.
	Hence Question \ref{Q5} is settled affirmatively and Question \ref{Q3} negatively.
	Moreover, with regard to Question \ref{Q1}, the sequences 
	$\left\{ G_n \right\}_{n=1} ^{\infty}$
	and 
	$\left\{ H_n \right\}_{n=1} ^{\infty}$
	exhibit that 
	$\chi^{{\rm FAT}}(G) - \chi(G)$ and $\chi(H) - \chi^{{\rm FAT}}(H)$
	are unbounded, already among connected graphs.
\end{remark}

\section{Concluding Remarks}

In the preceding section of this paper, in two distinct subsections, we examined Questions \ref{Q1}–\ref{Q5} separately
for connected and for disconnected graphs. In the present section, we dispense with that dichotomy and address these
Questions in their full generality. The following theorem and remark furnish concise answers to these Questions without any
restriction on whether the graphs are connected.

\begin{theorem} \label{General}
	For each integer $n \geq 3$, there exist graphs $G_1$
	and $G_2$ satisfying
	$$\begin{array}{cccc}
		\chi \left( G_1 \right) =1 , \  &  \chi ^{{\rm FAT}} \left( G_1 \right) = n, \ & \chi \left( G_2 \right) =n,\  & \chi ^{{\rm FAT}} \left( G_2 \right) = 1. 
	\end{array}$$ 
\end{theorem}

\medskip

\begin{proof}
	Let $G_1$ be $\bar{K}_n$, which is the edgeless graph on $n$ vertices.
	In view of the Proof of Theorem \ref{Disconnected1}, the graph $G_1$
	satisfies
	$$\begin{array}{llcrr}
		\chi \left( G_1 \right) = 1  &  &  {\rm and}  &  &  \chi ^{{\rm FAT}} \left( G_1 \right) = n .
	\end{array}$$
	
	Let $G_2$ be a graph with $n+1$ vertices, whose vertex set equals
	$$V \left( G_2 \right) := \left\{ a_1 , a_2 , \dots , a_{n+1} \right\} .$$
	Also, let $E \left( G_2 \right) := E_1 \cup E_2$, where $E_1$ and $E_2$
	are defined as the following:
	$$\begin{array}{llcrr}
		E_1 := \left\{ a_i a_j : \ \ 1 \leq i < j \leq n \right\}  &   &   {\rm and}  &   &  E_2 := \left\{ a_1 a_{n+1} \right\}.
	\end{array}$$
	Thus $G_2$ is obtained from the complete graph on $n$ vertices $a_1$, $a_2$, $\dots$, $a_n$,
	by adjoining a vertex $a_{n+1}$ that is pendant to $a_1$.
	
	The graph $G_2$ satisfies $\chi \left( G_2 \right) = n$.
	Since $G_2$ is a connected graph and $\delta \left( G_2 \right) = 1$,
	it follows that
	$$ \chi ^{{\rm FAT}} \left( G_2 \right) \leq \delta \left( G_2 \right) + 1 = 2 . $$
	We claim that
	$ \chi ^{{\rm FAT}} \left( G_2 \right) = 1 $.
	Suppose, contrary to the claim, that
	$ \chi ^{{\rm FAT}} \left( G_2 \right) \neq 1 $.
	Therefore,
	$ \chi ^{{\rm FAT}} \left( G_2 \right) = 2 $,
	and $G_2$ admits some FAT $2$-coloring with nonempty color classes
	$V_1$ and $V_2$ together with corresponding parameters $\alpha$ and $\beta$.
	Without loss of generality, one may assume that
	$a_{n+1} \in V_2$.
	Since $G_2$ is connected and $V_1$ and $V_2$ are nonempty,
	we must have
	$\alpha > 0$.
	This shows that
	$$ \alpha = \alpha \times 1 = \alpha . \deg_{G_2} \left( a_{n+1} \right) = e \left( a_{n+1} , V_1 \right) \in \{0,1\} .$$
	Accordingly,
	$$\begin{array}{llcrr}
		\alpha = e \left( a_{n+1} , V_1 \right)  = 1  &  &  {\rm and}  &  &  a_1 \in V_1 .
	\end{array}$$
	Thus, the fact that $ a_1 \in V_1 $ implies
	$$ e \left( a_1 , V_2 \right) = \alpha . \deg_{G_2} \left( a_1 \right) = \deg_{G_2} \left( a_1 \right)  =
	\left|  V  \left( G_2 \right)  \right| - 1 .$$
	Hence,
	$$\begin{array}{llcrr}
		V_1 = \left\{ a_1 \right\}  &  &  {\rm and}  &  &  V_2 = V \left(  G_2  \right) \setminus \left\{ a_1 \right\} .
	\end{array}$$
	We conclude that $a_3 \in V_2$. On the other hand, since $a_2 \in V_2$,
	it follows that
	$$ e \left( a_2 , V_1 \right) = \alpha . \deg_{G_2} \left( a_2 \right) = \deg_{G_2} \left( a_2 \right)  . $$
	Accordingly, all neighbors of $a_2$ must belong to $V_1$.
	In particular,
	$a_3 \in V_1$, which contradicts the former fact that $a_3 \in V_2$.
	The contradiction shows that the assumption was false, and therefore the claim holds.
	This completes the proof.    
\end{proof}

\medskip

We conclude with the following remark.
\begin{remark}
	According to Theorem \ref{General}, Question \ref{Q1} is answered negatively with respect to boundedness: 
	each of $\chi^{{\rm FAT}}(G) - \chi(G)$ and $\chi(H) - \chi^{{\rm FAT}}(H)$ is unbounded both above and below.
	Moreover,
	Theorem \ref{General} implies that
	Questions \ref{Q2} and \ref{Q3} are answered negatively,
	whereas Questions \ref{Q4} and \ref{Q5} are settled affirmatively.
\end{remark}

\vskip 0.4 true cm

\bibliographystyle{amsplain}
\def\cprime{$'$} \def\cprime{$'$}
\providecommand{\bysame}{\leavevmode\hbox to3em{\hrulefill}\thinspace}
\providecommand{\MR}{\relax\ifhmode\unskip\space\fi MR }
\providecommand{\MRhref}[2]{%
	\href{http://www.ams.org/mathscinet-getitem?mr=#1}{#2}
}
\providecommand{\href}[2]{#2}



\bigskip
\bigskip


{\footnotesize {\bf Saeed Shaebani}\; \\ {School of Mathematics and Computer Science}, {Damghan University,} {Damghan, Iran.}\\
{\tt Email: shaebani@du.ac.ir}\\

\end{document}